\documentclass[a4paper,12pt]{amsproc}
\usepackage{amsmath,amssymb,mathptmx,cite}
\usepackage{mathtools}

\makeatletter
\newcommand{\subjclassname@later}{\textup{2010} Mathematics Subject Classification}
\makeatother

\newtheorem{theorem}{Theorem}[section]
\newtheorem{lemma}{Lemma}[section]

\newtheorem{corollary}{Corollary}[section]

\theoremstyle{definition}
\newtheorem{definition}{Definition}[section]
\newtheorem{remark}{Remark}[section]

\numberwithin{equation}{section} \numberwithin{theorem}{section}

\DeclareMathOperator{\RE}{Re}

\allowdisplaybreaks
\begin{document}
\title[Second Hankel Determinant of Certain Univalent Functions]{Bounds for the Second Hankel Determinant \\
of Certain Univalent Functions }

\author[S. K. Lee]{See Keong Lee}
\address{School of Mathematical Sciences, Universiti Sains Malaysia,
11800 USM,  Penang,  Malaysia} \email{sklee@cs.usm.my }

\author[V.  Ravichandran]{ V.  Ravichandran}
\address{School of Mathematical Sciences, Universiti Sains Malaysia,
11800 USM,  Penang,  Malaysia\\ Department of Mathematics, University of Delhi, Delhi 110
007, India} \email{ vravi@maths.du.ac.in }

\author[S. Supramaniam]{Shamani Supramaniam}
\address{School of Mathematical Sciences,
Universiti Sains Malaysia,  11800 USM,  Penang,  Malaysia}
\email{sham105@hotmail.com}


\begin{abstract}The estimates for the second Hankel determinant $a_2a_4-a_3^2$ of analytic function $f(z)=z+a_2 z^2+a_3 z^3+\dotsb$ for which either $zf'(z)/f(z)$ or $1+zf''(z)/f'(z)$ is subordinate to certain analytic function are investigated. The estimates for the Hankel determinant for two other classes are also obtained. In particular, the estimates for the Hankel determinant of strongly starlike, parabolic starlike, lemniscate starlike functions are obtained.
\end{abstract}

\vspace*{-2cm}
\maketitle

\section{Introduction}
Let $\mathcal{A}$ denote the class of all analytic functions \begin{equation}\label{eqf} f(z)=z+a_2 z^{2}+a_{3}z^{3}+\dotsb\end{equation} defined on the open unit disk $\mathbb{D}:=\{ z \in \mathbb{C} : |z| < 1 \}$.
The \emph{Hankel determinants} $H_q (n)$, ($n=1,2,\dotsc$, $q=1,2,\dotsc$)  of the function $f$ are defined by \[H_q (n):=\left|\begin{array}{cccc}
                                                                a_n & a_{n+1} & \cdots & a_{n+q-1} \\
                                                                a_{n+1} & a_{n+2} & \cdots & a_{n+q} \\
                                                               \vdots & \vdots &  & \vdots \\
                                                                a_{n+q-1} & a_{n+q} & \cdots & a_{n+2q-2}
                                                              \end{array}
\right|, \quad (a_1=1).\]
Hankel determinants are useful, for example, in showing that a function of
bounded characteristic in $\mathbb{D}$, i.e.,  a function which is a ratio of two bounded analytic  functions, with its  Laurent series around the origin having integral coefficients,  is rational \cite{cantor}. For the use of Hankel determinant in the study of meromorphic functions, see \cite{wilson}, and various properties of these determinants can be found in  \cite[Chapter 4]{vein}.   In 1966, Pommerenke \cite{pom1} investigated the Hankel determinant  of areally mean  $p$-valent functions, univalent functions as well as for starlike functions. In \cite{pom2}, he proved that the Hankel determinants of  univalent functions satisfy  \[|H_q (n)|<Kn^{-(\frac{1}{2}+\beta)q+\frac{3}{2}} \quad (n=1,2,\dotsc,\ q=2,3,\dotsc),\] where $\beta>1/4000$ and $K$ depends only on $q$.  Later, Hayman \cite{hay} proved that $|H_2{(n)}|<A n^{1/2}$,  $(n=1,2,\dotsc;\ A\ \text{an absolute constant})$ for areally mean univalent functions. In \cite{noon1,noon2,noon3}, the  estimates for Hankel determinant for areally mean $p$-valent functions were investigated.  ElHosh obtained   bounds for Hankel determinants of univalent functions with positive Hayman index $\alpha$ \cite{elhosh} and of $k$-fold symmetric and close-to-convex functions \cite{elhosh1}. For bounds on the Hankel determinants of close-to-convex functions, see  \cite{noor1,noor2,noor3}. Noor studied the Hankel determinant of  Bazilevic functions in \cite{noorbaz} and  of functions with bounded boundary rotation in \cite{noorbr1,noorbr2,noorbr3,noorbr4}.  In the recent years, several authors have investigated bounds for the  Hankel determinant of functions  belonging to various subclasses of univalent and multivalent functions  \cite{arif,hayami1,hayami2,hayami3,janteng,mish,nmoh,murugu}. The
Hankel determinant  $H_2 (1)=a_3-a_2^2$ is the well known Fekete-Szeg\"o functional. For results related to this functional, see \cite{rma1, rma2}. The second Hankel determinant $H_2 (2)$ is given by $H_2(2)=a_2 a_4- a_3^2$.

An analytic function $f$ is\emph{ subordinate} to an analytic function $g$, written $f(z)\prec g(z)$, if  there is an analytic function $w:\mathbb{D}\rightarrow \mathbb{D}$ with $w(0)= 0$ satisfying $f(z)= g(w(z))$. Ma and Minda {\cite{mamin}} unified various subclasses of
starlike ($\mathcal{S^*}$) and convex functions ($\mathcal{C}$) by requiring that either of the quantity $zf'(z)/ f(z)$ or $1 + zf''(z)/ f'(z)$ is subordinate to a  function  $\varphi$ with positive real part in the unit disk $\mathbb{D}$,
$\varphi(0)= 1$, $\varphi'(0)> 0$,   $\varphi$ maps $\mathbb{D}$ onto a region starlike with respect to $1$ and symmetric with respect to the
real axis. He obtained distortion, growth and covering estimates as well as bounds for the initial coefficients of the unified classes.

The bounds for the second Hankel determinant $H_2 (2)=a_2 a_4- a_3^2$ are obtained for functions belonging to these subclasses of Ma-Minda starlike and convex functions in Section 2. In section 3, the problem is investigated for  two other related classes defined by subordination. In proving our results, we do not assume the univalence or starlikeness of $\varphi$ as they were required only in obtaining the distortion, growth estimates and the convolution theorems.   The classes introduced by subordination naturally include several well known classes of univalent functions and the results for some of these special classes are indicated as corollaries.

Let $\mathcal{P}$ be the class of \emph{functions with positive real part} consisting of all analytic functions $p:\mathbb{D}\rightarrow \mathbb{C}$ satisfying $p(0)=1$ and $\RE p(z)>0$. We need the following  results about the functions belonging to the class $\mathcal{P}$:

\begin{lemma} {\rm \cite{duren}}  \/ If  the function $p\in \mathcal{P}$ is  given by the series \begin{equation}\label{eqp}p(z)=1+c_1z+c_2z^2+c_3z^3+\dotsb,\end{equation} then the following sharp estimate holds: \begin{equation}\label{eqc}|c_n|\leq 2 \quad (n= 1, 2,\dotsc).\end{equation}
\end{lemma}

\begin{lemma}{\rm \cite{szego}} If  the function $p\in \mathcal{P}$ is  given by the series  \eqref{eqp}, then
\begin{align}\label{eqc2}2c_2&=c_1^2+x(4-c_1^2),\\
\label{eqc3}4c_3&=c_1^3+2(4-c_1^2)c_1x-c_1(4-c_1^2)x^2+2(4-c_1^2)(1-|x|^2)z,\end{align} for some $x, z$ with $|x|\leq1$ and $|z|\leq1$.
\end{lemma}


\section{Second Hankel determinant of Ma-Minda starlike/convex functions}

Various subclasses of starlike functions  are characterized by the quantity $zf'(z)/f(z)$ lying in some domain in the right half-plane. For example, $f$ is strongly  starlike of order $\beta$ if $zf'(z)/f(z)$  lies in a sector $|\arg w|<\beta \pi/2$ while it is starlike of order $\alpha$ if $zf'(z)/f(z)$ lies in the half-plane $\RE w>\alpha$. The various  subclasses of starlike functions were unified by subordination in \cite{mamin}.
The following definition of the class of Ma-Minda starlike functions is the same as the one in \cite{mamin} except for the omission of starlikeness assumption of $\varphi$.

\begin{definition}
Let $\varphi:\mathbb{D}\rightarrow \mathbb{C}$ be analytic and  the Maclaurin series  of    $\varphi $ is given by  \begin{equation}
  \label{varphi} \varphi(z)= 1 + B_1 z + B_2 z^2 + B_3 z^3 + \dotsb , \quad (B_1, B_2\in \mathbb{R},\ B_1 >0). \end{equation}
 The class $\mathcal{S^*}(\varphi)$ of \emph{Ma-Minda starlike functions with respect to $\varphi$} consists of functions $f\in\mathcal{A} $ satisfying the subordination \[\frac{zf'(z)}{ f(z)}\prec
\varphi( z).\]
\end{definition}

  For the function $\varphi$ given by $\varphi_\alpha(z):=(1+(1-2\alpha)z)/(1-z)$ , $0<\alpha\leq1$, the class $\mathcal{S^{*}}(\alpha):=\mathcal{S^*}\left(\varphi_\alpha\right)$ is the well-known class of starlike functions of order $\alpha$.
 Let \[ \varphi_{PAR}(z):=1+\frac{2}{\pi^2}\left(\log \frac{1+\sqrt{z}}{1-\sqrt{z}}\right)^2.\] Then the class
\[ \mathcal{S}^*_P:=\mathcal{S^*}(\varphi_{PAR})= \left\{f\in\mathcal{A}:\RE\left(\frac{zf'(z)}{f(z)}\right)> \left| \frac{zf'(z)}{f(z)} -1\right|\right\}  \] is the \emph{parabolic starlike} functions introduced by R\o nning \cite{ronn}. For a survey of parabolic starlike functions and the related class of uniformly convex functions, see \cite{ucv}. For $0<\beta\leq1$, the class
\[ \mathcal{S^{*}_{\beta}}:=\mathcal{S^*}\left(\Big(\frac{1+z}{1-z}\Big)^\beta\right)  =
  \left\{f\in\mathcal{A}: \left|\arg\left(\frac{zf'(z)}{f(z)}\right) \right|<\frac{\beta \pi}{2} \right\}\]
  is the familiar class of \emph{strongly starlike functions of order $\beta$}. The class
\[ \mathcal{S}^*_L:=\mathcal{S^*}(\sqrt{1+z})=\left\{f\in\mathcal{A}: \left|\left(\frac{zf'(z)}{f(z)}\right)^2-1\right|<1\right\}\]
is the class of \emph{lemniscate starlike} functions studied in \cite{sokol}.

\begin{theorem}\label{thhd2}
 Let the function $f \in \mathcal{S}^{*}(\varphi)$ be given by \eqref{eqf}.
 \begin{enumerate}
   \item   If $B_1$, $B_2$ and $B_3$ satisfy the conditions
\[ |B_2|\leq B_1,\quad 4B_1^4 -16B_1|B_3|+12B_2^2-6B_1|B_2|+9B_1^2\geq 0,\]  then  the second Hankel determinant satisfies
 \[|a_2 a_4 -a_3^2|\leq \frac{B_1^2}{4}.\]
   \item  If $B_1$, $B_2$ and $B_3$ satisfy the conditions \[|B_2|\geq B_1,\quad 4B_1^4 -16B_1|B_3|+12B_2^2-2B_1|B_2|+5B_1^2\leq 0,\]  or the conditions \[ |B_2|\leq B_1,\quad 4B_1^4 -16B_1|B_3|+12B_2^2-6B_1|B_2|+9B_1^2\leq 0,\]  then  the second Hankel determinant satisfies
 \[|a_2 a_4 -a_3^2|\leq \frac{1}{48}(-4B_1^4+16B_1|B_3|-12{B_2^2}+6B_1|B_2|+3B_1^2).\]
   \item  If $B_1$, $B_2$ and $B_3$ satisfy the conditions \[ |B_2|> B_1,\quad 4B_1^4 -16B_1|B_3|+12B_2^2-2B_1|B_2|+5B_1^2\geq 0,\] then  the second Hankel determinant satisfies
   \[|a_2 a_4 -a_3^2|\leq \frac{B_1^2}{12}\left(\frac{12B_1^4-48B_1|B_3|+40B_2^2-2B_1|B_2|+7B_1^2}{4B_1^4-16B_1|B_3|+12B_2^2+2B_1|B_2|+B_1^2}\right).\]
 \end{enumerate}
\end{theorem}

\begin{proof}
Since $f\in \mathcal{S}^{*}(\varphi)$,  there exists an analytic function $w$ with $w(0)=0$ and $|w(z)|<1$ in $\mathbb{D}$ such that
\begin{equation}\label{eqth2} \frac{zf'(z)}{ f(z)}= \varphi(w(z)).\end{equation}
Define the functions $p_1$ by
  \begin{equation*}
    \label{eqpu} p_1(z): = \frac{1 + w(z)}{1 - w(z)}=1 + c_1 z + c_2 z^2 +
     \cdots  \end{equation*} or equivalently,
  \begin{equation}
    \label{eqw} w(z)= \frac{p_1(z)- 1}{p_1(z)+ 1}=\frac{1}{2} \left( c_1 z
     + \left(c_2 - \frac{c_1^2}{2} \right)z^2 + \cdots \right).\end{equation}
Then $p_1$ is analytic in $\mathbb{D}$ with $p_1(0)=1$ and has positive real part in $\mathbb{D}$.
By using \eqref{eqw} together with {\eqref{varphi}},
it is evident that
  \begin{equation}
    \label{var1} \varphi \left(\frac{p_1(z)- 1}{p_1(z)+ 1} \right)=
    1 + \frac{1}{2} B_1 c_1 z + \left( \frac{1}{2} B_1 \left(c_2 -
    \frac{c_1^2}{2} \right)+ \frac{1}{4} B_2 c_1^2 \right) z^2 + \cdots.
  \end{equation}
Since
\begin{equation}\label{exps}\frac{zf'(z)}{f(z)}=1+a_2z+(-a_2^2+2a_3)z^2+(3a_4-3a_2a_3+a_2^3)z^3+\cdots,\end{equation}
it follows by \eqref{eqth2}, \eqref{var1}   and \eqref{exps} that
\begin{align*} a_2&=\frac{ B_1 c_1}{2},\\
a_3&=\frac{1}{8}\left[(B_1^2-B_1+B_2)c_1^2+2B_1c_2 \right],\\
a_4&=\frac{1}{48}[(-4B_2+2B_1+B_1^3-3B_1^2+3B_1B_2+2B_3)c_1^3+2(3B_1^2-4B_1+4B_2)c_1 c_2+8B_1 c_3].
\end{align*}
Therefore
\begin{align*}a_2a_4-a_3^2&=\frac{B_1}{96}\left[c_1^4\left(-\frac{B_1^3}{2}+\frac{B_1}{2}-B_2+2B_3-\frac{3B_2^2}{2B_1}\right)+2c_2c_1^2(B_2-B_1)+8B_1c_1c_3-6B_1c_2^2\right]. \end{align*}
Let
\begin{equation}\label{eqds2}
  \begin{array}{ll}
   d_1=8B_1, \quad d_2=2(B_2-B_1), \\[10pt]
   d_3=-6B_1, \quad d_4=-\frac{B_1^3}{2}+\frac{B_1}{2}-B_2+2B_3-\frac{3B_2^2}{2B_1},\\[10pt]
   T=\frac{B_1}{96}.
  \end{array}
\end{equation} Then
\begin{equation}\label{eqd2}
|a_2a_4-a_3^2|=T|d_1c_1c_3+d_2c_1^2c_2+d_3c_2^2+d_4c_1^4 |.
\end{equation}
Since the function  $p(e^{i\theta}z)$ $(\theta\in\mathbb{R})$ is in the class $\mathcal{P}$ for any $p\in\mathcal{P}$, there is no loss of generality in assuming $c_1>0$. Write $c_1=c$, $c\in[0,2]$. Substituting the values of $c_2$ and $c_3$ respectively from \eqref{eqc2} and \eqref{eqc3} in \eqref{eqd2}, it follows that
\begin{align*}
|a_2a_4-a_3^2|&=\frac{T}{4}\left|c^4(d_1+2d_2+d_3+4d_4)+2xc^2(4-c^2)(d_1+d_2+d_3)\right.\\&\quad \left.+(4-c^2)x^2(-d_1c^2+d_3(4-c^2))+2d_1c(4-c^2)(1-|x|^2)z\right|.
\end{align*}
Replacing  $|x|$ by $\mu$ and substituting the values of $d_1$, $d_2$, $d_3$ and $d_4$ from \eqref{eqds2}, yield
\begin{align}
|a_2a_4-a_3^2|&\leq\frac{T}{4}\Big[c^4\left(-2B_1^3+8|B_3|-6\frac{B_2^2}{B_1}\right)+4|{B_2}|\mu c^2(4-c^2)\notag\\
&\quad +\mu^2(4-c^2)(2B_1c^2+24B_1)+16B_1c(4-c^2)(1-\mu^2)\Big]\notag\\
&=T\Big[\frac{c^4}{4}\left(-2B_1^3+8|B_3|-6\frac{B_2^2}{B_1}\right)+4B_1c(4-c^2)+|{B_2}|(4-c^2)\mu c^2\label{eqF2}\\
&\quad +\frac{B_1}{2}\mu^2(4-c^2)(c-6)(c-2)\Big]\notag\\&\equiv F(c,\mu).\notag
\end{align}
Note that for $(c,\mu) \in [0,2] \times [0,1]$, differentiating $F(c,\mu)$ in \eqref{eqF2} partially with respect to $\mu$ yields
\begin{equation}\label{eqF2'}
\frac{\partial F}{\partial \mu}=T\left[|B_2|(4-c^2)+B_1\mu(4-c^2)(c-2)(c-6)\right].
\end{equation}
Then for $0<\mu<1$ and for any fixed $c$ with $0<c<2$, it is clear from \eqref{eqF2'} that $\frac{\partial F}{\partial \mu}>0$, that is, $F(c,\mu)$ is an increasing function of $\mu$. Hence for fixed $c\in [0,2]$, the maximum of $F(c,\mu)$ occurs at $\mu=1$, and
\[\max F(c,\mu)=F(c,1)\equiv G(c).\]
Also note that
\begin{align*}G(c)&=\frac{B_1}{96}\left[ \frac{c^4}{4}\left(-2B_1^3+8|B_3|-6\frac{B_2^2}{B_1}-|B_2|-\frac{B_1}{2}\right)+4c^2(|B_2|-B_1)+24B_1\right].\end{align*}
Let \begin{align}P&=\frac{1}{4}\left(-2B_1^3+8|B_3|-6\frac{B_2^2}{B_1}-|B_2|-\frac{B_1}{2}\right),\notag\\
Q&=4(|B_2|-B_1),\label{eqxyz2}\\
R&=24B_1.\notag\end{align} Since
\begin{equation}\label{ma}\max_{0\leq t\leq4} (Pt^2+Qt+R)=\left\{
    \begin{array}{ll}
      R, & \hbox{$Q\leq0$, $P\leq-\frac{Q}{4}$;} \\[10pt]
      16P+4Q+R, & \hbox{$Q\geq0$, $P\geq-\frac{Q}{8}$ or $Q\leq0$, $P\geq-\frac{Q}{4}$;} \\[10pt]
      \frac{4PR-Q^2}{4P}, & \hbox{$Q>0$, $P\leq-\frac{Q}{8}$,}
    \end{array}
  \right.
\end{equation}
we have
\[|a_2a_4-a_3^2|\leq \frac{B_1}{96}\left\{
    \begin{array}{ll}
      R, & \hbox{$Q\leq0$, $P\leq-\frac{Q}{4}$;} \\[10pt]
      16P+4Q+R, & \hbox{$Q\geq0$, $P\geq-\frac{Q}{8}$ or $Q\leq0$, $P\geq-\frac{Q}{4}$;} \\[10pt]
      \frac{4PR-Q^2}{4P}, & \hbox{$Q>0$, $P\leq-\frac{Q}{8}$}
    \end{array}
  \right.\]
 where $P,Q,R$ are given by \eqref{eqxyz2}.
\end{proof}

\begin{remark}
When $B_1=B_2=B_3=2$, Theorem \ref{thhd2} reduces to \cite[Theorem 3.1]{janteng}.
\end{remark}

\begin{corollary}
\begin{enumerate}\item[]
\item If $f\in\mathcal{S^{*}}(\alpha)$, for $0<\alpha\leq3/4$,  $|a_2a_4-a_3^2|\leq (1-\alpha)^2$. And for $3/4\leq\alpha\leq1$, $|a_2a_4-a_3^2|\leq (1-\alpha)^2[13-16(1-\alpha)^2]/12$.

  \item  If $f\in \mathcal{S}^*_L$, then $|a_2 a_4-a_3^2|\leq {1}/{16}=0.0625$.

 \item   If $f\in \mathcal{S}^*_P$, then $|a_2 a_4-a_3^2|\leq {16}/{\pi^4}\approx 0.164255$.

 \item  If $f\in\mathcal{S^{*}_{\beta}}$, then $|a_2a_4-a_3^2|\leq \beta^2$.
\end{enumerate}
\end{corollary}

\begin{definition} Let $\varphi:\mathbb{D}\rightarrow \mathbb{C}$ be analytic and $\varphi(z)$ is given as in \eqref{varphi}. The class $\mathcal{C}(\varphi)$ of \emph{Ma-Minda convex functions with respect to $\varphi$} consists of functions $f $ satisfying the subordination \[1 +
\frac{zf''(z)}{f'(z)}\prec \varphi(z).\]\end{definition}

\begin{theorem}\label{thhd3}
 Let the function $f \in  \mathcal{C}(\varphi)$ be given by \eqref{eqf}.
 \begin{enumerate}
   \item   If $B_1$, $B_2$ and $B_3$ satisfy the conditions
   \[B_1^2+4|B_2|-2B_1\leq 0,\quad B_1^4-B_1^2|B_2| -6B_1|B_3|+4B_2^2+4B_1^2\geq 0,\] then the second Hankel determinant satisfies
 \[|a_2 a_4 -a_3^2|\leq \frac{B_1^2}{36}.\]

 \item
 If $B_1$, $B_2$ and $B_3$ satisfy the conditions \[B_1^2+4|B_2|-2B_1\geq 0,\quad 2B_1^4-2B_1^2|B_2| -12B_1|B_3|+8B_2^2+4B_1|B_2|+B_1^3+6B_1^2\leq 0,\]  or the conditions \[B_1^2+4|B_2|-2B_1\leq 0,\quad  B_1^4-B_1^2|B_2| -6B_1|B_3|+4B_2^2+4B_1^2\leq 0,\]  then the second Hankel determinant satisfies
 \[|a_2 a_4 -a_3^2|\leq \frac{1}{144}(-B_1^4+B_1^2|B_2|+6B_1|B_3|-4{B_2^2}).\]
 \item
 If $B_1$, $B_2$ and $B_3$ satisfy the conditions \[ B_1^2+4|B_2|-2B_1> 0,\quad 2B_1^4-2B_1^2|B_2| -12B_1|B_3|+8B_2^2+4B_1|B_2|+B_1^3+6B_1^2\geq 0,\] then the second Hankel determinant satisfies
  \[|a_2 a_4 -a_3^2|\leq  \frac{B_1^2}{576}\left(\frac{17B_1^4-8B_1^2|B_2|-96B_1|B_3|+80B_2^2+12B_1^3+48B_1|B_2|+36B_1^2}{B_1^4-B_1^2|B_2|-6B_1|B_3|+4B_2^2+B_1^3+4B_1|B_2|+2B_1^2}\right).\]
\end{enumerate}
\end{theorem}

\begin{proof}
Since $f\in \mathcal{C}(\varphi)$,  there exists an analytic function $w$ with $w(0)=0$ and $|w(z)|<1$ in $\mathbb{D}$ such that
\begin{equation}\label{eqth3} 1+\frac{zf''(z)}{ f'(z)}= \varphi(w(z)).\end{equation}
Since
\begin{equation}\label{expc}1+\frac{zf''(z)}{f'(z)}=1+2a_2z+(-4a_2^2+6a_3)z^2+(8a_2^3-18a_2a_3+12a_4)z^3+\cdots,\end{equation}
equations \eqref{var1}, \eqref{eqth3}  and \eqref{expc} yield
\begin{align*} a_2&=\frac{ B_1 c_1}{4},\\
a_3&=\frac{1}{24}\left[(B_1^2-B_1+B_2)c_1^2+2B_1c_2 \right],\\
a_4&=\frac{1}{192}[(-4B_2+2B_1+B_1^3-3B_1^2+3B_1B_2+2B_3)c_1^3+2(3B_1^2-4B_1+4B_2)c_1 c_2+8B_1 c_3].
\end{align*}
Therefore
\begin{align*}a_2a_4-a_3^2&=\frac{B_1}{768}\left[c_1^4\left(-\frac{4}{3}B_2+\frac{2}{3}B_1 -\frac{1}{3}B_1^3-\frac{1}{3}B_1^2+\frac{1}{3}B_1B_2+2B_3-\frac{4}{3}\frac{B_2^2}{B_1}\right)\right.\\&\quad\left. +\frac{2}{3}c_2c_1^2(B_1^2-4B_1+4B_2)
+8B_1c_1c_3-\frac{16}{3}B_1c_2^2\right]. \end{align*}
By writing
\begin{equation}\label{eqds3}
  \begin{array}{ll}
   d_1=8B_1, \quad d_2=\frac{2}{3}(B_1^2-4B_1+4B_2), \\[10pt]
   d_3=-\frac{16}{3}B_1, \quad d_4=-\frac{4}{3}B_2+\frac{2}{3}B_1 -\frac{1}{3}B_1^3-\frac{1}{3}B_1^2+\frac{1}{3}B_1B_2+2B_3-\frac{4}{3}\frac{B_2^2}{B_1},\\[10pt]
   T=\frac{B_1}{768},
  \end{array}
\end{equation} we have
\begin{equation}\label{eqd3}
|a_2a_4-a_3^2|=T|d_1c_1c_3+d_2c_1^2c_2+d_3c_2^2+d_4c_1^4 |.
\end{equation}
Similar as in Theorems \ref{thhd2}, it follows from \eqref{eqc2} and \eqref{eqc3} that
\begin{align*}
|a_2a_4-a_3^2|&=\frac{T}{4}\left|c^4(d_1+2d_2+d_3+4d_4)+2xc^2(4-c^2)(d_1+d_2+d_3)\right.\\&\quad \left.+(4-c^2)x^2(-d_1c^2+d_3(4-c^2))+2d_1c(4-c^2)(1-|x|^2)z\right|.
\end{align*}
Replacing  $|x|$ by $\mu$ and then substituting the values of $d_1$, $d_2$, $d_3$ and $d_4$ from \eqref{eqds3} yield
\begin{align}
|a_2a_4-a_3^2|&\leq\frac{T}{4}\Big[c^4\left(-\frac{4}{3}B_1^3+\frac{4}{3}B_1 B_2+8B_3-\frac{16}{3}\frac{B_2^2}{B_1}\right)+2\mu c^2(4-c^2)\left(\frac{2}{3}B_1^2+\frac{8}{3}B_2\right)\notag\\
&\quad +\mu^2(4-c^2)\left(\frac{8}{3}B_1c^2+\frac{64}{3}B_1\right)+16B_1c(4-c^2)(1-\mu^2)\Big]\notag\\
&=T\Big[\frac{c^4}{3}\left(-B_1^3+B_1|B_2|+6|B_3|-4\frac{B_2^2}{B_1}\right)+4B_1c(4-c^2)+\frac{1}{3}\mu c^2(4-c^2)(B_1^2+4|{B_2}|)\label{eqF3}\\
&\quad +\frac{2B_1}{3}\mu^2(4-c^2)(c-4)(c-2)\Big]\notag\\&\equiv F(c,\mu).\notag
\end{align}
 Again, differentiating $F(c,\mu)$ in \eqref{eqF3} partially with respect to $\mu$ yield
\begin{equation}\label{eqF3'}
\frac{\partial F}{\partial \mu}=T\left[\frac{c^2}{3} (4-c^2)(B_1^2+4|{B_2}|)+\frac{4B_1}{3}\mu(4-c^2)(c-4)(c-2)\right].
\end{equation}
It is clear from \eqref{eqF3'} that $\frac{\partial F}{\partial \mu}>0$. Thus $F(c,\mu)$ is an increasing function of $\mu$ for $0<\mu<1$ and for any fixed $c$ with $0<c<2$. So the maximum of $F(c,\mu)$ occurs at $\mu=1$ and \[\max F(c,\mu)=F(c,1)\equiv G(c).\]
Note that
\begin{align*}G(c)&=T\left[ \frac{c^4}{3}\left(-B_1^3+B_1|B_2|+6|B_3|-4\frac{B_2^2}{B_1}-B_1^2-4|B_2|-2B_1\right)\right.\\&\quad \left.+\frac{4}{3}c^2(B_1^2+4|B_2|-2B_1) +\frac{64}{3}B_1\right].\end{align*}
Let \begin{align}P&=\frac{1}{3}\left(-B_1^3+B_1|B_2|+6|B_3|-4\frac{B_2^2}{B_1}-B_1^2-4|B_2|-2B_1\right),\notag\\
Q&=\frac{4}{3}(B_1^2+4|B_2|-2B_1),\label{eqxyz3}\\
R&=\frac{64}{3}B_1,\notag
\end{align} By using \eqref{ma},
we have
\[|a_2a_4-a_3^2|\leq \frac{B_1}{768}\left\{
    \begin{array}{ll}
      R, & \hbox{$Q\leq0$, $P\leq-\frac{Q}{4}$;} \\[10pt]
      16P+4Q+R, & \hbox{$Q\geq0$, $P\geq-\frac{Q}{8}$ or $Q\leq0$, $P\geq-\frac{Q}{4}$;} \\[10pt]
      \frac{4PR-Q^2}{4P}, & \hbox{$Q>0$, $P\leq-\frac{Q}{8}$}
    \end{array}
  \right.\]
where $P,Q,R$ are given in \eqref{eqxyz3}.
\end{proof}

\begin{remark}For the choice of $\varphi(z)=(1+z)/(1-z)$,    Theorem \ref{thhd3} reduces to \cite[Theorem 3.2]{janteng}.
\end{remark}

\section{Further results on the second Hankel determinant }

\begin{definition} Let $\varphi:\mathbb{D}\rightarrow \mathbb{C}$ be analytic and $\varphi(z)$ as given in \eqref{varphi}.
Let $0\leq\gamma\leq1$ and $\tau\in\mathbb{C}\setminus \{0\}$. A function $f\in \mathcal{A}$ is in the class $\mathcal{R}_{\gamma}^{\tau}(\varphi)$ if it satisfies the following subordination:
\[1+\frac{1}{\tau}(f'(z)+\gamma zf''(z)-1)\prec \varphi(z).\]
\end{definition}

\begin{theorem}\label{hd1}
 Let $0\leq\gamma\leq1$, $\tau\in\mathbb{C}\setminus \{0\}$ and the function $f$ as in \eqref{eqf} is in the class $\mathcal{R}_{\gamma}^{\tau}(\varphi)$. Also, let \[p=\frac{8}{9}\frac{(1+\gamma)(1+3\gamma)}{(1+2\gamma)^2}.\]
\begin{enumerate}
  \item If $B_1$, $B_2$ and $B_3$ satisfy the conditions \[ 2|B_2|(1-p)+B_1(1-2p)\leq 0,\quad |B_1B_3 -pB_2^2|-pB_1^2 \leq 0,\] then the second Hankel determinant satisfies
 \[|a_2 a_4 -a_3^2|\leq \frac{|\tau|^2 B_1^2}{9(1+2\gamma)^2}.\]
  \item  If $B_1$, $B_2$ and $B_3$ satisfy the conditions \[ 2|B_2|(1-p)+B_1(1-2p)\geq 0,\quad 2|B_1B_3-pB_2^2|-2(1-p)B_1|B_2|-B_1\geq0,\] or the conditions \[2|B_2|(1-p)+B_1(1-2p)\leq 0,\quad |B_1B_3 -pB_2^2|-B_1^2 \geq 0,\]  then the second Hankel determinant satisfies
 \[|a_2 a_4 -a_3^2|\leq \frac{|\tau|^2}{8(1+\gamma)(1+3\gamma)}|B_3B_1-pB_2^2|.\]
  \item  If $B_1$, $B_2$ and $B_3$ satisfy the conditions \[ 2|B_2|(1-p)+B_1(1-2p)> 0,\quad2|B_1B_3-pB_2^2|-2(1-p)B_1|B_2|-B_1^2\leq0,\] then the second Hankel determinant satisfies
  \begin{align*} |a_2 a_4 -a_3^2|&\leq  \frac{|\tau|^2 B_1^2}{32(1+\gamma)(1+3\gamma)}\left(\frac{\splitfrac{4p|B_3B_1-pB_2^2|-4(1-p)B_1[|B_2|(3-2p)+B_1]}{-4B_2^2(1-p)^2-B_1^2(1-2p)^2}}{|B_3B_1-pB_2^2|-(1-p)B_1(2|B_2|+B_1)}\right).\end{align*}
\end{enumerate}
\end{theorem}

\begin{proof}
For $f\in \mathcal{R}_{\gamma}^{\tau}(\varphi)$,  there exists an analytic function $w$ with $w(0)=0$ and $|w(z)|<1$ in $\mathbb{D}$ such that
\begin{equation}\label{eqth1} 1+\frac{1}{\tau}(f'(z)+\gamma zf''(z)-1)= \varphi(w(z)).\end{equation}
Since $f$ has the  Maclaurin series given by
\eqref{eqf}, a computation shows that
\begin{equation}
\label{expr} 1+\frac{1}{\tau}(f'(z)+\gamma zf''(z)-1)=1+\frac{2a_2(1+\gamma)}{\tau}z+\frac{3a_3(1+2\gamma)}{\tau}z^2+\frac{4a_4(1+3\gamma)}{\tau}z^3+\cdots.
\end{equation}
It follows from \eqref{eqth1}, \eqref{var1} and \eqref{expr} that
\begin{align*} a_2&=\frac{\tau B_1c_1}{4(1+\gamma)},\\
a_3&=\frac{\tau B_1}{12(1+2\gamma)}\left[2c_2+c_1^2\left(\frac{B_2}{B_1}-1\right)\right],\\
a_4&=\frac{\tau}{32(1+3\gamma)}[B_1(4c_3-4c_1c_2+c_1^3)+2B_2c_1(2c_2-c_1^2)+B_3c_1^3].\end{align*}
Therefore
\begin{align*}a_2a_4-a_3^2&=\frac{\tau^2B_1c_1}{128(1+\gamma)(1+3\gamma)}\left[B_1(4c_3-4c_1c_2+c_1^3)+2B_2c_1(2c_2-c_1^2)+B_3c_1^3\right]
\\&\quad-\frac{\tau^2B_1^2}{144(1+2\gamma)^2}\left[4c_2^2+c_1^4\left(\frac{B_2}{B_1}-1\right)^2+4c_2c_1^2\left(\frac{B_2}{B_1}-1\right)\right]
\\&=\frac{\tau^2B_1^2}{128(1+\gamma)(1+3\gamma)}\left\{\left[(4c_1c_3-4c_1^2c_2+c_1^4)+\frac{2B_2c_1^2}{B_1}(2c_2-c_1^2)+\frac{B_3}{B_1}c_1^4\right]\right.
\\&\quad \left.-\frac{8}{9}\frac{(1+\gamma)(1+3\gamma)}{(1+2\gamma)^2}\left[4c_2^2+c_1^4\left(\frac{B_2}{B_1}-1\right)^2+4c_2c_1^2\left(\frac{B_2}{B_1}-1\right)\right]\right\}, \end{align*}
which yields
\begin{align}
|a_2a_4-a_3^2|&=T\left|4c_1c_3+c_1^4\left[1-2\frac{B_2}{B_1}-p\left(\frac{B_2}{B_1}-1\right)^2+\frac{B_3}{B_1}\right]-4p c_2^2 \right.\notag\\&\quad \left. -4c_1^2c_2\left[1-\frac{B_2}{B_1}+p\left(\frac{B_2}{B_1}-1\right)\right]\right|\label{eqpt},
\end{align}
where
\[T=\frac{|\tau|^2 B_1^2}{128(1+\gamma)(1+3\gamma)}\quad \text{and}\quad p=\frac{8}{9}\frac{(1+\gamma)(1+3\gamma)}{(1+2\gamma)^2}.\]
It can be easily verified that $p\in\left[\frac{64}{81},\frac{8}{9}\right]$ for $0\leq\gamma\leq1$.

Let
\begin{equation}\label{eqds}
  \begin{array}{ll}
   d_1=4, \quad d_2=-4\left[1-\frac{B_2}{B_1}+p\left(\frac{B_2}{B_1}-1\right)\right], \\[10pt]
   d_3=-4p, \quad d_4=1-2\frac{B_2}{B_1}-p\left(\frac{B_2}{B_1}-1\right)^2+\frac{B_3}{B_1}.
  \end{array}
\end{equation} Then \eqref{eqpt} becomes
\begin{equation}\label{eqd}
|a_2a_4-a_3^2|=T|d_1c_1c_3+d_2c_1^2c_2+d_3c_2^2+d_4c_1^4 |.
\end{equation}
It follows that
\begin{align*}
|a_2a_4-a_3^2|&=\frac{T}{4}\left|c^4(d_1+2d_2+d_3+4d_4)+2xc^2(4-c^2)(d_1+d_2+d_3)\right.\\&\quad \left.+(4-c^2)x^2(-d_1c^2+d_3(4-c^2))+2d_1c(4-c^2)(1-|x|^2)z\right|.
\end{align*}
An application of triangle inequality, replacement of $|x|$ by $\mu$ and substituting the values of $d_1$, $d_2$, $d_3$ and $d_4$ from \eqref{eqds} yield
\begin{align}
|a_2a_4-a_3^2|&\leq\frac{T}{4}\Big[4c^4\left|\frac{B_3}{B_1}-p\frac{B_2^2}{B_1^2}\right|+8\left|\frac{B_2}{B_1}\right|\mu c^2(4-c^2)(1-p)\notag\\&\quad +(4-c^2)\mu^2(4c^2+4p(4-c^2))+8c(4-c^2)(1-\mu^2)\Big]\notag\\
&=T\Big[c^4\left|\frac{B_3}{B_1}-p\frac{B_2^2}{B_1^2}\right|+2c(4-c^2)+2\mu\left|\frac{B_2}{B_1}\right|c^2(4-c^2)(1-p)\label{eqF}\\
&\quad +\mu^2(4-c^2)(1-p)(c-\alpha)(c-\beta)\Big]\notag\\&\equiv F(c,\mu)\notag
\end{align}where $\alpha=2$, $\beta=2p/(1-p)>2$.

Similarly as in the previous proofs, it can be shown that $F(c,\mu)$ is an increasing function of $\mu$ for $0 < \mu < 1$. So for
fixed $c\in [0,2]$, let
\[\max F(c,\mu)=F(c,1)\equiv G(c),\]
which is
\begin{align*}G(c)&=T\left\{ c^4\left[\left|\frac{B_3}{B_1}-p\frac{B_2^2}{B_1^2}\right|-(1-p)\left(2\left|\frac{B_2}{B_1}\right|+1\right)\right]\right.\\
&\quad \left.+4c^2\left[2\left|\frac{B_2}{B_1}\right|(1-p)+1-2p\right]+16p\right\}.\end{align*}
Let \begin{align}P&=\left|\frac{B_3}{B_1}-p\frac{B_2^2}{B_1^2}\right|-(1-p)\left(2\left|\frac{B_2}{B_1}\right|+1\right),\notag\\
Q&=4\left[2\left|\frac{B_2}{B_1}\right|(1-p)+1-2p\right],\label{eqxyz1}\\
R&=16p.\notag\end{align}
Using \eqref{ma}, we have
\[|a_2a_4-a_3^2|\leq T\left\{
    \begin{array}{ll}
      R, & \hbox{$Q\leq0$, $P\leq-\frac{Q}{4}$;} \\[10pt]
      16P+4Q+R, & \hbox{$Q\geq0$, $P\geq-\frac{Q}{8}$ or $Q\leq0$, $P\geq-\frac{Q}{4}$;} \\[10pt]
      \frac{4PR-Q^2}{4P}, & \hbox{$Q>0$, $P\leq-\frac{Q}{8}$}
    \end{array}
  \right.\]
where $P,Q,R$ are given in \eqref{eqxyz1}.
\end{proof}


\begin{remark}For the choice $\varphi(z):=(1+Az)/(1+Bz)$ with $-1\leq B<A\leq1$,  Theorem \ref{hd1} reduces to \cite[Theorem 2.1]{bansal}.
\end{remark}

\begin{definition} Let $\varphi:\mathbb{D}\rightarrow \mathbb{C}$ be analytic and $\varphi(z)$ as given in \eqref{varphi}.
For fixed real number $\alpha$, function $f\in \mathcal{A}$ is in the class $\mathcal{G}_{\alpha}(\varphi)$ if it satisfies the following subordination:
\[(1-\alpha)f'(z)+\alpha\left(1 +
\frac{zf''(z)}{f'(z)}\right)\prec \varphi(z).\]
\end{definition}

 Al-Amiri  and Reade \cite{alamiri} introduced  the class  $ \mathcal{G}_{\alpha}:=\mathcal{G}_{\alpha}((1+z)/(1-z))$ and they have shown that $\mathcal{G}_{\alpha}\subset \mathcal{S}$ for $\alpha<0$. Univalence of the functions in the class $\mathcal{G}_{\alpha}$ was  also investigated in \cite{ssingh,singh}.   Singh \emph{et al.} also obtained the bound for the second Hankel determinant of functions in $\mathcal{G}_{\alpha}$. The following theorem provides a bound for the second Hankel determinant of the functions in the class $\mathcal{G}_{\alpha}(\varphi)$.

\begin{theorem}\label{thhd4}
Let the function $f$ given by \eqref{eqf} be in the class $\mathcal{G}_{\alpha}(\varphi)$, $0\leq\alpha\leq1$. Also, let \[p=\frac{8}{9}\frac{(1+2\alpha)}{(1+\alpha)}.\]
\begin{enumerate}

  \item If $B_1$, $B_2$ and $B_3$ satisfy the conditions \[ B_1^2\alpha(3-2p)+2|B_2|(1+\alpha-p)+B_1(1+\alpha-2p)\leq 0,\] \[B_1^4\alpha(2\alpha-1-p\alpha)+\alpha B_1^2|B_2|(3-2p)+(\alpha+1)B_1|B_3|-p(B_1^2 +B_2^2)\leq 0,\] then  the second Hankel determinant satisfies
 \[|a_2 a_4 -a_3^2|\leq \frac{ B_1^2}{9(1+\alpha)^2}.\]

  \item  If $B_1$, $B_2$ and $B_3$ satisfy the conditions \[B_1^2\alpha(3-2p)+2|B_2|(1+\alpha-p)+B_1(1+\alpha-2p)\geq 0,\] \[2B_1^4\alpha(2\alpha-1-p\alpha)+2\alpha B_1^2|B_2|(3-2p)-B_1^3\alpha(3-2p)\]\[+2(\alpha+1)B_1|B_3|-2(1+\alpha -p)B_1|B_2|-(1+\alpha)B_1^2-2pB_2^2\geq0,\] or \[B_1^2\alpha(3-2p)+2|B_2|(1+\alpha-p)+B_1(1+\alpha-2p)\leq 0,\] \[B_1^4\alpha(2\alpha-1-p\alpha)+\alpha B_1^2|B_2|(3-2p)+(\alpha+1)B_1|B_3|-p(B_1^2 +B_2^2)\geq 0,\] then  the second Hankel determinant satisfies
 \begin{align*}&|a_2 a_4 -a_3^2| \leq \frac{ B_1^4\alpha(2\alpha-1-p\alpha)+\alpha B_1^2|B_2|(3-2p)+(\alpha+1)B_1|B_3|+p(B_2^2-B_1^2)}{8(1+\alpha)(1+2\alpha)}.\end{align*}

  \item If $B_1$, $B_2$ and $B_3$ satisfy the conditions \[B_1^2\alpha(3-2p)+2|B_2|(1+\alpha-p)+B_1(1+\alpha-2p)> 0,\]\[ 2B_1^4\alpha(2\alpha-1-p\alpha)+2\alpha B_1^2|B_2|(3-2p)-B_1^3\alpha(3-2p)\]\[+2(\alpha+1)B_1|B_3|-2(1+\alpha -p)B_1|B_2|-(1+\alpha)B_1^2-2pB_2^2\leq0,\] then   the second Hankel determinant satisfies
 \begin{align*}
& |a_2 a_4 -a_3^2|\leq \frac{ B_1^2}{32(1+\alpha)(1+2\alpha)}\left[4p -\frac{[B_1^2\alpha(3-2p)+2|B_2|(1+\alpha-p)+B_1(1+\alpha-2p)]^2}{ \splitfrac{B_1^4\alpha(2\alpha-1-p\alpha)+\alpha B_1^2|B_2|(3-2p)-B_1^3\alpha(3-2p)}{+(\alpha+1)B_1|B_3|-(1+\alpha-p)B_1(2|B_2|+1)-pB_2^2}}\right].
 \end{align*}
\end{enumerate}

\end{theorem}

\begin{proof}
For  $f\in \mathcal{G}_{\alpha}(\varphi)$,  a calculation shows that

\begin{align}
|a_2a_4-a_3^2|&=T\Big|4(1+\alpha)B_1 c_1c_3+c_1^4\Big[-3\alpha B_1^2+\alpha(2\alpha-1)B_1^3+B_1(1+\alpha)+3\alpha B_1B_2\notag\\&\quad +(1+\alpha)(B_3-2B_2)-p\frac{(\alpha B_1^2-B_1+B_2)^2}{B_1}\Big]-4pB_1 c_2^2 \notag\\&\quad  +2c_1^2c_2\left[-2(1+\alpha)B_1+3\alpha B_1^2+2(1+\alpha)B_2-2p(\alpha B_1^2-B_1+B_2)\right]\Big|\label{eqpt4}
\end{align}
where
\[T=\frac{B_1}{128(1+\alpha)(1+2\alpha)}\quad \text{and}\quad p=\frac{8}{9}\frac{(1+2\alpha)}{(1+\alpha)}.\]
It can be easily verified that for $0\leq\alpha\leq1$, $p\in\left[\frac{8}{9},\frac{4}{3}\right]$.
Let
\begin{equation}\label{eqds4}
  \begin{array}{ll}
   d_1=4(1+\alpha)B_1, \\[10pt] d_2=2\left[-2(1+\alpha)B_1+3\alpha B_1^2+2(1+\alpha)B_2-2p(\alpha B_1^2-B_1+B_2)\right], \\[10pt]
   d_3=-4pB_1, \\[10pt] d_4=-3\alpha B_1^2+\alpha(2\alpha-1)B_1^3+B_1(1+\alpha)+3\alpha B_1B_2+(1+\alpha)(B_3-2B_2)-p\frac{(\alpha B_1^2-B_1+B_2)^2}{B_1},\\[10pt]
  \end{array}
\end{equation} Then
\begin{equation}\label{eqd4}
|a_2a_4-a_3^2|=T|d_1c_1c_3+d_2c_1^2c_2+d_3c_2^2+d_4c_1^4 |.
\end{equation}
Similar as in earlier theorems, it follows that
\begin{align}
|a_2a_4-a_3^2|&=\frac{T}{4}\left|c^4(d_1+2d_2+d_3+4d_4)+2xc^2(4-c^2)(d_1+d_2+d_3)\right.\\&\quad \left.+(4-c^2)x^2(-d_1c^2+d_3(4-c^2))+2d_1c(4-c^2)(1-|x|^2)z\right|\notag\\
&\leq T\Big[{c^4}\left[B_1^3\alpha(2\alpha-1-p\alpha)+\alpha B_1|B_2|(3-2p)+(\alpha+1)|B_3|-p\frac{B_2^2}{B_1}\right]\notag\\&\quad+\mu c^2(4-c^2)[B_1^2\alpha(3-2p)+2|B_2|(1+\alpha-p)]+2c(4-c^2)B_1(1+\alpha)\label{eqF4}\\
&\quad +\mu^2(4-c^2)B_1(1+\alpha-p)(c-2)\left(c-\frac{2p}{1+\alpha-p}\right)\Big]\notag\\&\equiv F(c,\mu),\notag
\end{align}
 and for fixed $c\in [0,2]$,
$\max F(c,\mu)=F(c,1)\equiv G(c)$ with
\begin{align*}G(c)&=T\Bigg[c^4 \Big[B_1^3\alpha(2\alpha-1-p\alpha)+\alpha B_1|B_2|(3-2p)-B_1^2\alpha(3-2p)+(\alpha+1)|B_3| \\&\quad-(1+\alpha-p)(2|B_2|+B_1)-p\frac{B_2^2}{B_1}\Big]+4c^2 [B_1^2\alpha(3-2p)+2|B_2|(1+\alpha-p)\\& \quad +B_1(1+\alpha-2p)]+16pB_1 \Bigg].\end{align*}
Let \begin{align}P&=B_1^3\alpha(2\alpha-1-p\alpha)+\alpha B_1|B_2|(3-2p)-B_1^2\alpha(3-2p)+(\alpha+1)|B_3|\notag\\&\quad\quad{}-(1+\alpha-p)(2|B_2|+B_1)-p\frac{B_2^2}{B_1}\notag\\
Q&=4\left[B_1^2\alpha(3-2p)+2|B_2|(1+\alpha-p)+B_1(1+\alpha-2p)\right],\label{eqxyz4}\\
R&=16pB_1,\notag\end{align}
 By using \eqref{ma}, we have
\[|a_2a_4-a_3^2|\leq T\left\{
    \begin{array}{ll}
      R, & \hbox{$Q\leq0$, $P\leq-\frac{Q}{4}$;} \\[10pt]
      16P+4Q+R, & \hbox{$Q\geq0$, $P\geq-\frac{Q}{8}$ or $Q\leq0$, $P\geq-\frac{Q}{4}$;} \\[10pt]
      \frac{4PR-Q^2}{4P}, & \hbox{$Q>0$, $P\leq-\frac{Q}{8}$}
    \end{array}
  \right.\]
where $P,Q,R$ are given in \eqref{eqxyz4}.
\end{proof}

\begin{remark} For $\alpha=0$,  Theorem \ref{thhd4} reduces to Theorem \ref{thhd3}.  For $0\leq\alpha<1$, let $\varphi(z):=(1+(1-2\alpha)z)/(1-z)$. For this function $\varphi$,
  $B_1=B_2=B_3=2(1-\alpha)$. In this case, Theorem \ref{thhd4} reduces to \cite[Theorem 3.1]{gupta}.
\end{remark}

\end{document}